\documentclass[a4paper,12pt]{amsart}
\usepackage[margin=2.5cm]{geometry}
\usepackage[utf8]{inputenc}
\usepackage[T1]{fontenc}
\usepackage[charter]{mathdesign}
\usepackage{microtype}
\usepackage{hyperref}
\hypersetup{
bookmarks=true,
pdfpagemode=UseNone,
pdfstartview=FitH,
pdfdisplaydoctitle=true,
pdfborder={0 0 0}, 
pdftitle={Lower bound in the Roth theorem for amenable groups},
pdfauthor={Qing Chu and Pavel Zorin-Kranich},
pdfsubject={Mathematics Subject Classification (2010): 37A15},
pdfkeywords={multiple recurrence, amenable group},
pdflang=en-US
}

\usepackage[safeinputenc=true,backref=true,style=alphabetic,firstinits,useprefix=true,url=false]{biblatex}

\usepackage{tikz}

\numberwithin{equation}{section}
\newtheorem{theorem}[equation]{Theorem}
\newtheorem{lemma}[equation]{Lemma}
\newtheorem{proposition}[equation]{Proposition}
\newtheorem{corollary}[equation]{Corollary}
\theoremstyle{definition}
\newtheorem{definition}[equation]{Definition}
\theoremstyle{remark}

\newcommand{\Clim}{\mathrm{C-}\lim}

\newcommand{\UClim}{\mathrm{UC-}\lim}

\newcommand{\probmeas}{\mathcal{P}}
\newcommand{\ud}{\overline{d}}
\newcommand*{\MPS}[1][k]{\mathcal{C}_{#1}}
\newcommand*{\eMPS}[1][k]{\mathcal{C}^{\mathrm{erg}}_{#1}}
\newcommand*{\Inv}[1][k]{\mathcal{I}_{#1}}
\newcommand*{\IF}[1]{I_{#1}} 
\newcommand*{\KF}[1]{K_{#1}} 
\newcommand{\AP}{\mathrm{AP}}
\newcommand{\folner}{F\o{}lner}
\newcommand{\reiter}{Reiter}
\newcommand{\cesaro}{Ces\`aro}
\newcommand{\id}{\mathrm{id}}
\newcommand{\Id}{\mathrm{Id}}
\newcommand*{\inv}{^{-1}}
\newcommand*{\E}{\mathbb{E}}
\newcommand{\dif}{\mathrm{d}}
\newcommand{\R}{\mathbb{R}}
\newcommand{\C}{\mathbb{C}}
\newcommand{\N}{\mathbb{N}}
\newcommand{\Z}{\mathbb{Z}}
\def\<{\left\langle}
\def\>{\right\rangle}

\begin{document}
\subjclass[2010]{37A15}
\title{Lower bound in the Roth theorem for amenable groups}
\date{\today}
\author{Qing Chu}
\address[Qing Chu]{%
  Department of Mathematics\\
  The Ohio State University\\
  100 Math Tower\\
  231 West 18th Avenue\\
  Columbus, OH 43210-1174\\
  USA%
}
\email{chu.270@math.osu.edu}
\author{Pavel Zorin-Kranich}
\address[Pavel Zorin-Kranich]
{Institute of Mathematics\\
Hebrew University, Givat Ram\\
Jerusalem, 91904, Israel}
\email{pzorin@math.huji.ac.il}
\urladdr{http://math.huji.ac.il/~pzorin/}
\keywords{Roth theorem, amenable group}
\begin{abstract}
Let $T_{1}$ and $T_{2}$ be two commuting probability measure-preserving actions of a countable amenable group such that the group spanned by these actions acts ergodically.
We show that $\mu(A\cap T_{1}^{g}A\cap T_{1}^{g}T_{2}^{g}A) > \mu(A)^{4}-\epsilon$ on a syndetic set for any measurable set $A$ and any $\epsilon>0$.
The proof uses the concept of a sated system introduced by Austin.
\end{abstract}
\maketitle

\section*{Introduction}
In this article we are concerned with the correlation sequence $c_{g}=\mu(A\cap T_{1}^{g}A \cap T_{1}^{g}T_{2}^{g}A)$, where $A$ is a positive measure subset of a regular probability space $(X,\mu)$ and $T_{1},T_{2}$ are measure-preserving actions of a countable amenable group $G$ that commute in the sense that $T_{1}^{g}T_{2}^{h}=T_{2}^{h}T_{1}^{g}$ a.e.\ for any $g,h\in G$.
In the case $G=\Z$, $T_{1}=T_{2}$, the ergodic theoretic version \cite{MR0498471} of the Roth theorem \cite{MR0051853} states that $c_{g}$ is positive on a syndetic set.
The ergodic theoretic version of the corners theorem \cite{MR0369299} says that the same continues to hold without the assumption $T_{1}=T_{2}$.
This has been extended to general countable amenable groups in \cite{MR1481813}.

In the ergodic case, namely if the only sets that are both $T_{1}$- and $T_{2}$-invariant are those with measure $0$ or $1$, one can say more.
If $G=\Z$ and $T_{1}=T_{2}$, it is known the correlation function $c_{g}$ is not only positive on a syndetic set, but in fact bounded below by $\mu(A)^{3}-\epsilon$ on a syndetic set for any $\epsilon>0$ \cite[Theorem 1.2]{MR2138068}.
If $G=\Z$ but $T_{1}$ and $T_{2}$ are not necessarily equal, then the best result up to date is that $c_{g}$ is bounded below by $\mu(A)^{4}-\epsilon$ on a syndetic set for any $\epsilon>0$ \cite[Theorem 1.1]{MR2794947}.
The exponent $4$ in the latter result cannot be improved to $3$, see Theorem~\ref{thm:small-correlation}.
Our purpose is to obtain a similar lower bound for general countable amenable groups (in fact a similar result holds for locally compact second countable amenable groups, although its formulation is more involved, see Theorem~\ref{thm:recurrence}).
\begin{theorem}
\label{thm:erg}
Let $G$ be a countable amenable group and $T_{1},T_{2}$ be commuting measure-preserving left $G$-actions on a probability space $(X,\mu)$.
Suppose that the group spanned by $T_{1}$ and $T_{2}$ acts ergodically on $(X,\mu)$.
Then for every measurable set $A\subset X$ and every $\epsilon>0$ the set
\[
R_{\epsilon} := \{ g\in G : \mu(A\cap T_{1}^{g}A \cap T_{1}^{g}T_{2}^{g}A) > \mu(A)^{4} - \epsilon \}
\]
is both left and right syndetic.
\end{theorem}
By a version of the Furstenberg correspondence principle \cite[Theorem 4.17]{MR1776759} this result has a combinatorial interpretation.
Recall that the \emph{upper Banach density} of a subset $A\subset G$ is defined by
$\ud(A) := \sup_{F} \limsup_{N} |A\cap F_{N}|/|F_{N}|$,
where the supremum is taken over all left \folner{} sequences in $G$.
\begin{theorem}
\label{thm:combi}
Let $G$ be a countable amenable group and let $E\subset G\times G$.
Then for every $\epsilon>0$ the set
\[
\{g\in G: \ud(E\cap (g,\id)E\cap (g,g)E) \geq \ud(E)^{4} - \epsilon\}
\]
is both left and right syndetic.
\end{theorem}
It is known that no result similar to Theorem~\ref{thm:erg} can hold for four commuting measure-preserving actions \cite[Theorem 1.3]{MR2138068}.
The question as to whether a power lower bound exists for three commuting actions remains open even for $G=\Z$.
A proof that the corresponding \cesaro{} limit is positive has been recently made available by Austin \cite{arXiv:1309.4315}.

At a first glance it might appear that at least non-triviality of $R_{\epsilon}$ in Theorem~\ref{thm:erg} would follow from the previously known case $G=\Z$ upon restriction to a suitable cyclic subgroup of $G$.
However, there are at least two obstructions to such reasoning.
Firstly, $G$ can be a torsion group.
Secondly, even if there exist copies of $\Z$ in $G$, the restrictions of $(T_{1},T_{2})$ to these copies can be non-ergodic, preventing one from applying the $G=\Z$ case.

Our arguments rely on a \emph{magic extension} of the system $(X,\mu,T_{1},T_{2})$ obtained as a \emph{sated extension} in the sense of Austin \cite{MR2941687}.
The concept of a magic extension was introduced by Host \cite{MR2539560}.
In course of the proof we obtain new proofs of two further results about commuting actions of amenable groups.
The first result is an extension to arbitrary left \folner{} sequences of the convergence theorem for cubic averages due to Griesmer \cite{2008arXiv0812.1968G} (Corollary~\ref{cor:cubic-convergence}).
This has been previously shown by Bergelson and Leibman (private communication).
The second result is the $k=3$ case of the convergence theorem for multiple ergodic averages \cite[Theorem 1.1(2)]{arxiv:1111.7292} (Proposition~\ref{prop:conv-k=3}).
Another proof (for arbitrary $k$) using sated extensions appeared after the completion of this work in \cite[Theorem A]{arXiv:1309.4315}.

\section{Preliminaries}
\subsection{\folner{} and \reiter{} sequences}
Let $G$ be a locally compact $\sigma$-compact (lcsc) group with a left Haar measure $m$.
If the group $G$ is amenable, then by \cite[Theorem 4.16]{MR961261} it admits a \emph{left \folner{} sequence}, that is, a sequence of non-null compact sets $F_{N}\subset G$ such that $m(gF_{N}\Delta F_{N})/m(F_{N}) \to 0$ uniformly for $g$ in compact subsets of $G$.

The space $M(G)$ of complex Radon measures on $G$ is a Banach $*$-algebra with the convolution $\int f \dif(\mu * \nu) = \int_{g,h} f(gh) \dif\mu(g)\dif\nu(h)$, the involution $\int f \dif\mu^{*} = \int f(g\inv) \dif\bar\mu(g)$, and the total variation norm (see \cite{MR1397028} for more details).
Let $\probmeas(G)\subset M(G)$ be the set of probability Radon measures.
A sequence $(F_{N})\subset \probmeas(G)$ is called a \emph{left \reiter{} sequence} if $\|h*F_{N}-F_{N}\| \to 0$ as $N\to\infty$ for every $h\in \probmeas(G)$.
Right and two-sided \reiter{} sequences are defined analogously.
If $(F_{N})$ is a left \reiter{} sequence and $h\in \probmeas(G)$, $(F'_{N})\subset \probmeas(G)$, then $(h*F_{N})$ and $(F_{N}*F'_{N})$ are again left \reiter{} sequences, as follows from associativity of convolution.
Note that $(F_{N})$ is a left \reiter{} sequence if and only if $(F_{N}^{*})$ is a right \reiter{} sequence.

Given a left \folner{} sequence $F=(F_{N})$, the sequence $(m(F_{N})\inv 1_{F_{N}} m)$ is a left \reiter{} sequence, which will be denoted by the same symbol $F$.
The main conceptual reason to work with \reiter{} sequences rather than \folner{} sequences is that two-sided \reiter{} sequences also exist in non-unimodular groups, as opposed to two-sided \folner{} sequences.
Since for our purposes both concepts work equally well, we chose to stick with the more general one.

For a fixed \reiter{} sequence $F$ we write $\Clim_{g} u_{g} := \lim_{N} \int u_{g} \dif F_{N}(g)$ for the \cesaro{} limit along $F$ if this limit exists.
If the \cesaro{} limit exists for every \reiter{} sequence, then it does not depend on the \reiter{} sequence.
In this case we call it the uniform \cesaro{} limit and denote it by $\UClim_{g} u_{g}$.
Recall the van der Corput lemma.
\begin{lemma}[{\cite[Lemma 4.2]{MR1481813}}]
\label{lem:vdC}
Let $G$ be a lcsc group with a left \reiter{} sequence $F$.
Suppose that $g\mapsto u_{g}$ is a bounded measurable function to a Hilbert space.
Then
\[
\limsup_{m} \big\| \int_{g} u_{g} \dif F_{m}(g) \big\|^{2}
\leq
\inf_{H\in \probmeas(G)} \limsup_{m} \int_{g}\int_{l}\int_{h} \<u_{hg},u_{lg}\> \dif H(h) \dif H(l) \dif F_{m}(g).
\]
\end{lemma}
\begin{proof}
By definition of a \reiter{} sequence we have for every $H\in \probmeas(G)$
\[
\lim_{m} \big\|\int_{g} u_{g} \dif (F_{m}-H*F_{m})(g) \big\| = 0.
\]
Hence it suffices to estimate
\begin{multline*}
\limsup_{m} \big\| \int_{g}\int_{h} u_{hg} \dif H(h) \dif F_{m}(g) \big\|^{2}
\leq
\limsup_{m} \big( \int_{g} \big\| \int_{h} u_{hg} \dif H(h) \big\| \dif F_{m}(g) \big)^{2}\\
\leq
\limsup_{m} \int_{g} \big\| \int_{h} u_{hg} \dif H(h) \big\|^{2} \dif F_{m}(g)
=
\limsup_{m} \int_{g}\int_{l}\int_{h} \<u_{hg},u_{lg}\> \dif H(h)\dif H(l) \dif F_{m}(g),
\end{multline*}
where we used the triangle and the Cauchy--Schwarz inequalities.
\end{proof}
\begin{corollary}
\label{cor:vdC}
Under the assumptions of Lemma~\ref{lem:vdC} we have
\[
\limsup_{m} \big\| \int_{g} u_{g} \dif F_{m}(g) \big\|^{2}
\leq
\inf_{n} \limsup_{m} \int_{g} \int_{h} \<u_{g},u_{hg}\> \dif F'_{n}(h) \dif F_{m}(g)
\]
for some two-sided \reiter{} sequence $F'$.
\end{corollary}
\begin{proof}
Substitute $H=F_{n}$ in Lemma~\ref{lem:vdC}.
By the Fubini theorem and the \reiter{} property we have
\begin{multline*}
\limsup_{m} \iiint \<u_{hg},u_{lg}\> \dif F_{n}(h)\dif F_{n}(l)\dif F_{m}(g)\\
=
\limsup_{m} \iiint \<u_{hg},u_{lg}\> \dif (\delta_{h\inv}*F_{m})(g) \dif F_{n}(h) \dif F_{n}(l)\\
=
\limsup_{m} \int_{g} \iint \<u_{g},u_{lh\inv g}\> \dif F_{n}(h)\dif F_{n}(l)\dif F_{m}(g).
\end{multline*}
The inner integrand can be written as
\begin{multline*}
\iint \<u_{g},u_{lh\inv g}\> \dif F_{n}(l) \dif F_{n}(h)
=
\iint \<u_{g},u_{lh g}\> \dif F_{n}(l) \dif F_{n}^{*}(h)
=
\int \<u_{g},u_{h g}\> \dif (F_{n}*F_{n}^{*})(h),
\end{multline*}
and we obtain the conclusion with $F'_{n} = F_{n} * F_{n}^{*}$.
\end{proof}

\subsection{Category of measure-preserving systems}
From now on let $G$ be a lcsc amenable group with a left \reiter{} sequence $F_{N}$.
\begin{definition}
The category $\MPS$ of $k$-tuples of commuting measure-preserving actions consists of the following data.
The objects are the tuples $(X,\mathcal{X},\mu,T_{1},\dots,T_{k})$, where $(X,\mathcal{X},\mu)$ is a regular Borel probability space and $T_{1},\dots,T_{k}$ are continuous commuting measure-preserving $G$-actions.
The morphisms are the continuous factor maps, that is, continuous measure-preserving maps that intertwine the respective $G$-actions.
\end{definition}
The restriction to continuous actions on regular Borel spaces is not substantial.
Indeed, suppose that we are given a measurable measure-preserving $G$-action on a separable measure space $X$.
The associated unitary antirepresentation on $L^{2}(X)$ is weakly measurable, and therefore strongly continuous by \cite[22.20(b)]{MR551496}.
In view of \cite[Theorem 7.5.5]{MR548006} this implies that the action admits a topological model on a compact metric space.

Measure-preserving actions on $X$ induce anti-actions on the spaces $L^{p}(X)$ that are denoted by the same symbol.

\subsection{Conditionally almost periodic and weakly mixing functions}
The following result is folklore and goes back to Furstenberg in the case $G=\Z$.
A proof for \folner{} sequences in general lcsc amenable groups will appear in \cite{robertson}; the same proof works without further changes for \reiter{} sequences.
The very last assertion is a pure Hilbert space result whose proof may be found in \cite{MR0174705}.
\begin{theorem}
\label{thm:A+W}
Let $G$ be a lcsc amenable group, $T$ be a measure-preserving $G$-action on a regular Borel probability space $X$ and $X\to Y$ be a factor map.
Then we have
\[
L^{2}(X) = A(X|Y,T) \oplus W(X|Y,T),
\]
where
\begin{enumerate}
\item the space $A(X|Y,T)$ is spanned by the finite rank $T$-invariant $L^{\infty}(Y)$-submodules of $L^{\infty}(X)$,\footnote{All that we need to know is that both the $Y$-measurable and the $T$-invariant functions are in $A(X|Y,T)$.} and
\item the space $W(X|Y,T)$ consists of the functions $f$ such that every $h\in L^{\infty}(X)$ and some/every left \reiter{} sequence $F$ we have
\[
\lim_{n} \int_{g} \| \E(hT^{g}f|Y) \|^{2} \dif F_{n}(g) = 0.
\]
\end{enumerate}
Moreover, for any two factor maps $X_{1}\to Y$ and $X_{2}\to Y$ we have
\[
A(X_{1}\times_{Y} X_{2}|Y,T) = A(X_{1}|Y,T) \otimes_{Y} A(X_{2}|Y,T).
\]
Finally, if $Y$ is the trivial factor, then $A(X|Y,T)$ is the closed linear span of the finite-dimensional $T$-invariant subspaces.
\end{theorem}

\subsection{Couplings of measure spaces}
We recall how to construct couplings of regular measure spaces.
\begin{lemma}
\label{lem:coupling}
Let $(X_{i},\mathcal{X}_{i},\mu_{i})_{i\in I}$ be (inner and outer) regular Borel probability spaces and let $\mu$ be a finitely additive positive function on the semiring of the sets of the form $\prod_{i\in I} A_{i}$, where $A_{i}\in\mathcal{X}_{i}$ and $A_{i}\neq X_{i}$ only for finitely many $i\in I$.
Suppose that $\mu(\prod_{i\in I} A_{i})\leq\min_{i}\mu_{i}(A_{i})$ for any $A_{i}$ as above.
Then $\mu$ admits a unique extension to a Borel probability measure on $(X,\mathcal{X}):=(\prod_{i\in I}X_{i},\bigotimes_{i\in I}\mathcal{X}_{i})$.
\end{lemma}
\begin{proof}
We will show that $\mu$ is in fact $\sigma$-additive.
To this end suppose that $\prod_{i} A_{i} = \uplus_{m\in\N} \prod_{i} A_{m,i}$ is a disjoint union with $A_{i},A_{m,i}\in\mathcal{X}_{i}$.
By finite additivity we have $\mu(\prod_{i} A_{i}) \geq \sum_{m\in\N} \mu(\prod_{i} A_{m,i})$.
In order to prove the converse inequality we use the regularity of the measures $\mu_{i}$.
Let $\delta>0$.
By inner regularity there exist compact subsets $C_{i}\subset A_{i}$ such that $C_{i}=X_{i}$ if $A_{i}=X_{i}$ and $\sum_{i\in I}\mu(C_{i})>\sum_{i}\mu(A_{i})-\delta$.
It follows that $\mu(\prod_{i} C_{i})>\mu(\prod_{i}A_{i})-\delta$.
By outer regularity there exist open subsets $U_{m,i}\supset A_{m,i}$ such that $\sum_{i}\mu(U_{m,i})<\sum_{i}\mu(A_{m,i})+\delta/2^{m}$.
It follows that $\mu(\prod_{i}U_{m,i})<\mu(\prod_{i}A_{m,i})+\delta/2^{m}$.

By construction the compact set $\prod_{i}C_{i}$ is covered by the open sets $\prod_{i}U_{m,i}$.
Hence there exists a finite subset $M\subset\N$ such that the corresponding open sets cover the whole compact set.
By additivity of $\mu$ this implies
\begin{multline*}
\mu(\prod_{i}A_{i})
< \mu(\prod_{i} C_{i}) + \delta
\leq \sum_{m\in M} \mu(\prod_{i} U_{m,i}) + \delta
\leq \sum_{m\in\N} \mu(\prod_{i} U_{m,i}) + \delta\\
\leq \sum_{m\in\N} (\mu(\prod_{i} A_{m,i}) + \delta/2^{m}) + \delta
= \sum_{m\in\N} \mu(\prod_{i} A_{m,i}) + 2\delta.
\end{multline*}
Since $\delta$ was arbitrary, this shows that $\mu$ is $\sigma$-subadditive.
By the Carath\'{e}odory theorem the function $\mu$ has a unique extension to a probability measure on the space $(X,\mathcal{X})$.
\end{proof}

\section{Sated systems and cubic averages}
We will use notation and vocabulary from Austin's thesis \cite{MR2941687}.
A subclass of $\MPS$ is called \emph{idempotent} if it contains the trivial system and is closed under measure-theoretic isomorphisms, inverse limits, and joinings.
\begin{lemma}[{\cite[Definition 2.2.3]{MR2941687}}]
\label{lem:max-I-factor}
Let $\Inv$ be an idempotent subclass of $\MPS$.
Then we have a functor on $\MPS$, which we denote by the same symbol $\Inv$, such that for each object $X$ of $\MPS$ the object $\Inv X$ is the maximal factor of $X$ contained in $\Inv$ (in particular, such maximal factor exists).
\end{lemma}
Let $(X,\mathcal{X},\mu)$ be a measure space and $\mathcal{B}\subset\mathcal{X}$ be a sub-$\sigma$-algebra.
Two sub-$\sigma$-algebras $\mathcal{B}\subset\mathcal{B}_{1},\mathcal{B}_{2}\subset\mathcal{X}$ are called \emph{relatively independent over $\mathcal{B}$} if for every $f\in L^{2}(\mathcal{B}_{1})$ such that $f\perp \mathcal{B}$ we have $f\perp\mathcal{B}_{2}$.
The notion of relative independence is in fact symmetric in $\mathcal{B}_{1}$ and $\mathcal{B}_{2}$; we refer to \cite[Appendix]{MR2373263} for an exposition of several further characterizations of relative independence.

Suppose that $\mathcal{C}$ is a subclass of $\MPS$ and $\Inv$ is an idempotent class.
A system $X$ in $\mathcal{C}$ is called \emph{$\Inv$-sated in $\mathcal{C}$} if for every extension $X'\to X$ with $X'$ in $\mathcal{C}$ the factors $X$ and $\Inv X'$ are relatively independent over $\Inv X$.
Note that Austin considers satedness only with $\mathcal{C} = \MPS$; we will have to work with a subclass in order to preserve ergodicity.
Recall that the \emph{inverse limit} of a sequence $(X_{i\in\N})$ in $\MPS$ with factor maps $\pi_{i+1,i}:X_{i+1}\to X_{i}$ is a system $X$ in $\MPS$ with factor maps $\pi_{i}:X\to X_{i}$ such that $\pi_{i+1,i}\circ\pi_{i+1}=\pi_{i}$ and $\mathcal{X} = \vee_{i\in\N} \pi_{i}\inv(\mathcal{X}_{i})$.
The next result is a variant of \cite[Theorem 2.3.2]{MR2941687}.
\begin{theorem}
\label{thm:sated-ext}
Suppose that $\mathcal{C}$ is a subclass of $\MPS$ that is closed under inverse limits of sequences and $\Inv$ is an idempotent subclass of $\MPS$.
Then for every $X$ in $\mathcal{C}$ there exists an extension $X'$ that is $\Inv$-sated in $\mathcal{C}$.
\end{theorem}
Here and later denote by $\IF{i}$ the $\sigma$-algebra of $T_{i}$-invariant sets.
We will use Theorem~\ref{thm:sated-ext} with the subclass $\Inv$ of $\MPS$ consisting of systems $X$ such that $\mathcal{X} = \IF{1}\vee\dots\vee \IF{k}$ and either with $\mathcal{C}=\MPS$ or with the subclass $\eMPS$ of jointly ergodic tuples of measure-preserving actions.
It is clear that the class $\Inv$ is idempotent.

The starting point of our investigation is a weak convergence result for cubic averages of dimension $2$.
For brevity we will write $T_{\epsilon}^{g}=\prod_{i\in\epsilon}T_{i}^{g}$ and $T_{\epsilon}^{\vec g}=\prod_{i\in\epsilon}T_{i}^{g_{i}}$ for $\epsilon\subseteq\{1,\dots,k\}$ and $g\in G$, $\vec g = (g_{1},\dots,g_{k}) \in G^{k}$.
Thus for instance $T_{\{1,2\}}^{g}=T_{1}^{g}T_{2}^{g}$ and $T_{\{1,2\}}^{\vec g} = T_{1}^{g_{1}}T_{2}^{g_{2}}$.
We will omit braces from subscripts if no confusion is possible, for example $T_{\{1,2\}}^{g}=T_{1,2}^{g}$.
\begin{lemma}
\label{lem:k=2-cube-measure}
Suppose that $X\in\MPS[2]$.
Then for any $f_{\epsilon}\in L^{\infty}(X)$, $\epsilon\subseteq\{1,2\}$, and any left \reiter{} sequences $\Phi,\Psi$ the limit
\begin{equation}
\label{eq:mu-square}
\lim_{n} \iiint \prod_{\epsilon} T_{\epsilon}^{\vec g} f_{\epsilon} \dif\mu \dif\Phi_{n}(g_{1}) \dif\Psi_{n}(g_{2})
\end{equation}
exists, and in particular it does not depend on $\Phi,\Psi$.
If in addition $X$ is $\Inv[2]$-sated, then the limit vanishes provided that $f_{1,2} \perp \Inv[2]X$.
\end{lemma}
\begin{proof}
Suppose first that $X$ is $\Inv[2]$-sated and recall that we assume the underlying measure space $(X,\mathcal{X},\mu)$ to be regular.
Let $B$ be a countable dense subalgebra of $\mathcal{X}$ and pick a subsequence of $\Phi\times\Psi$, which we denote by the same symbol, such that
\[
\mu^{\square}(\prod_{\epsilon} A_{\epsilon})
:= \lim_{n} \iiint \prod_{\epsilon} T_{\epsilon}^{\vec g} 1_{A_{\epsilon}} \dif\mu \dif\Phi_{n}(g_{1}) \dif\Psi_{n}(g_{2})
\]
exists for any $A_{\epsilon}\in B$.
The limit on the right-hand side of the above display is bounded by $\min_{\epsilon}\mu(A_{\epsilon})$, and it follows that $\mu^{\square}$ extends to a function on the semiring of sets of the form $\prod_{\epsilon} A_{\epsilon}$, $A_{\epsilon}\in\mathcal{X}$, by the same formula.
It is easy to see that $\mu^{\square}$ is finitely additive.

By Lemma~\ref{lem:coupling} the function $\mu^{\square}$ has a unique extension to a probability measure on $X^{4}$.
Since $\mu^{\square}$ is clearly invariant under the side transformations $T_{\square,1}$ and $T_{\square,2}$ given by
\[
T_{\square,i} = \times_{\epsilon\subseteq\{1,2\}} T_{\epsilon \cap \{i\}},
\]
the uniqueness implies that the extension, which we again denote by the symbol $\mu^{\square}$, is invariant with respect to these transformations.
Thus $(X^{4},\mathcal{X}^{\otimes 4},\mu^{\square},T_{\square,1},T_{\square,2})$ is an extension of $X$ under the projection $\pi_{1,2}$.
Moreover, $\int \otimes_{\epsilon} f_{\epsilon} \dif\mu^{\square}$ is given by the formula \eqref{eq:mu-square}.

Suppose now that $f_{1,2} \perp \Inv[2]X$.
By the satedness assumption the function $f_{1,2}\circ\pi_{1,2}$ on the cube extension is orthogonal to $\IF{\square,1}\vee\IF{\square,2}$ under $\mu^{\square}$.
On the over hand, if $\epsilon\subsetneq\{1,2\}$, then $f_{\epsilon}\circ \pi_{\epsilon}$ is $\IF{\square,i}$-measurable for any $i\not\in\epsilon$, so that
\[
0 = \int \prod_{\epsilon} f_{\epsilon}\circ\pi_{\epsilon} \dif\mu^{\square}
= \lim_{n} \iiint \prod_{\epsilon} T_{\epsilon}^{\vec g} f_{\epsilon} \dif\mu \dif\Phi_{n}(g_{1}) \dif\Psi_{n}(g_{2}).
\]
Since this limit does not depend on the subsequence of $\Phi\times\Psi$ that was chosen at the beginning, a subsubsequence argument shows that this limit in fact exists and vanishes for the original \reiter{} sequences.

On the other hand, in the case that $f_{1,2}$ is $\Inv[2]X$-measurable by density and linearity it suffices to consider $f_{1,2}=h_{1}h_{2}$, where each $h_{i}$ is $T_{i}$-invariant.
In this case we obtain
\begin{multline*}
\lim_{n} \iiint \prod_{\epsilon} T_{\epsilon}^{\vec g} f_{\epsilon} \dif\mu \dif\Phi_{n}(g_{1}) \dif\Psi_{n}(g_{2})\\
=
\lim_{n} \iiint T_{1}^{g_{1}} (f_{1}h_{2}) T_{2}^{g_{2}}(f_{2}h_{1}) f_{\emptyset} \dif\mu \dif\Phi_{n}(g_{1}) \dif\Psi_{n}(g_{2})\\
=
\int \E(f_{1}h_{2}|\IF{1}) \E(f_{2}h_{1}|\IF{2}) f_{\emptyset} \dif\mu
\end{multline*}
by the mean ergodic theorem.
This limit is manifestly independent of the \reiter{} sequences.

In the general case of a not necessarily $\Inv[2]$-sated system $X$ we use Theorem~\ref{thm:sated-ext} to pass to a sated extension and note that the existence of the limit \eqref{eq:mu-square} for functions on this extension implies the existence of that limit for functions on $X$.
\end{proof}

Thus we obtain a measure $\mu^{\square}$ on $X^{4}$ and two measure-preserving $G$-actions on $(X^{4},\mu^{\square})$ such that the resulting measure-preserving system is an extension of $X$.
This explicitly constructed extension allows us to exploit satedness.
Specifically, we aim at obtaining systems with the following property.
\begin{definition}
We call a system $X\in\MPS[2]$ \emph{magic} if $A(X|\IF{2}, T_{1}) = \IF{1} \vee \IF{2}$ (recall that $A(X|\IF{2}, T_{1})$ was defined in Theorem~\ref{thm:A+W}).
\end{definition}
An equivalent notion has been first introduced by Host \cite{MR2539560} for commutative $G$ (in fact he introduced a corresponding notion for $k$-tuples of commuting $\Z$-actions for every $k\in\N$).
The next proposition is our main tool for exploiting information about characteristic factors.
\begin{proposition}
\label{prop:KI-in-sated}
Suppose that $X\in\MPS[2]$ is $\Inv[2]$-sated.
Then $X$ is magic.
\end{proposition}
\begin{proof}
The inclusion $A(X|\IF{2}, T_{1}) \supseteq \IF{1} \vee \IF{2}$ holds in any measure-preserving system.
For the converse consider $f_{1}\perp \IF{1} \vee \IF{2}$.
Then for every  $f_{2}\in L^{\infty}(X)$ we have
\[
\limsup_{n} \int_{h} \|\E(f_{2}T_{1}^{h}f_{1}|\IF{2})\|^{2} \dif F_{n}(h)
=
\limsup_{n} \int_{h} \lim_{m} \int_{g} \int f_{2} T_{1}^{h} f_{1} \cdot T_{2}^{g}(f_{2} T_{1}^{h} f_{1}) \dif\mu \dif F_{m}(g) \dif F_{n}(h).
\]
This vanishes by Lemma~\ref{lem:k=2-cube-measure}.
Thus $f_{1}\perp A(X|\IF{2}, T_{1})$.
\end{proof}

In the remaining part of this section we extend Griesmer's cubic convergence result \cite[Theorem 1.4(1)]{2008arXiv0812.1968G} to arbitrary left \reiter{} sequences.
\begin{lemma}
\label{lem:k=2-char-factors}
Let $X\in\MPS[2]$ and $f_{\epsilon}\in L^{\infty}(X)$, $\epsilon\subseteq\{1,2\}$.
Suppose that $f_{1}\perp A(X|\IF{2},T_{1})$ or $f_{2}\perp A(X|\IF{1},T_{2})$.
Then for any left \reiter{} sequences $\Phi,\Psi$ on $G$ we have
\[
\lim_{n}\iint \prod_{\epsilon\subseteq\{1,2\}}T_{\epsilon}^{\vec g} f_{\epsilon} \dif\Phi_{n}(g_{1}) \dif\Psi_{n}(g_{2}) = 0
\quad \text{in } L^{2}(X).
\]
\end{lemma}
\begin{proof}
Since the $\epsilon=\emptyset$ term is a bounded function that does not depend on $g_{1},g_{2}$, we may discard it.
We apply Corollary~\ref{cor:vdC} to the map $(g_{1},g_{2})\mapsto \prod_{\epsilon\neq\emptyset}T_{\epsilon}^{\vec g} f_{\epsilon}$.
To show that it converges to zero in the \cesaro{} sense along $\Phi\times\Psi$ it thus suffices to show that
\[
\liminf_{n} \limsup_{m} \iiint \prod_{\epsilon\neq\emptyset}T_{\epsilon}^{\vec g} (f_{\epsilon} T_{\epsilon}^{\vec h} f_{\epsilon}) \dif\mu \dif(\Phi_{m}\times\Psi_{m})(g_{1},g_{2}) \dif(\Phi'_{n}\times\Psi'_{n})(h_{1},h_{2}) = 0.
\]
By Lemma~\ref{lem:k=2-cube-measure} the limit superior in $m$ is actually a limit and it does not depend on $\Phi,\Psi$.
Thus we may replace $\Phi,\Psi$ by two-sided \reiter{} sequences (this is how we remove the two-sidedness assumption from Griesmer's convergence result).
The double limit equals
\[
\liminf_{n} \lim_{m} \iiint T_{1}^{g_{1}\inv}T_{2}^{g_{2}\inv} \prod_{\epsilon\neq\emptyset}T_{\epsilon}^{\vec g} (f_{\epsilon} T_{\epsilon}^{\vec h} f_{\epsilon}) \dif\mu \dif(\Phi_{m}\times\Psi_{m})(g_{1},g_{2}) \dif(\Phi'_{n}\times\Psi'_{n})(h_{1},h_{2}),
\]
and by the mean ergodic theorem this equals
\[
\liminf_{n} \iint \E(f_{1}T_{1}^{h_{1}}f_{1}|\IF{2}) \E(f_{2} T_{2}^{h_{2}}f_{2}|\IF{1}) f_{1,2} T_{1,2}^{\vec h}f_{1,2} \dif\mu \dif(\Phi'_{n}\times\Psi'_{n})(h_{1},h_{2}).
\]
By the Cauchy--Schwarz inequality this is bounded by
\[
\liminf_{n} \|f_{1,2}\|_{\infty}^{2} \int \|\E(f_{1}T_{1}^{h_{1}}f_{1}|\IF{2})\|_{2} \dif\Phi'_{n}(h_{1}) \int \|\E(f_{2} T_{2}^{h_{2}}f_{2}|\IF{1})\|_{2} \dif\Psi'_{n}(h_{2}),
\]
and this vanishes by the assumption.
\end{proof}

\begin{corollary}
\label{cor:cubic-convergence}
Suppose that $X\in\MPS[2]$.
Then for any $f_{\epsilon}\in L^{\infty}(X)$, $\epsilon\subseteq\{1,2\}$, and any left \reiter{} sequences $\Phi,\Psi$ the limit
\[
\lim_{n} \iint \prod_{\epsilon\subseteq\{1,2\}} T_{\epsilon}^{\vec g} f_{\epsilon} \dif\Phi_{n}(g_{1}) \dif\Psi_{n}(g_{2})
\]
exists in $L^{2}(X)$ and does not depend on $\Phi,\Psi$.
\end{corollary}
\begin{proof}
By Theorem~\ref{thm:sated-ext} we may assume that $X$ is $\Inv[2]$-sated.
Then also the system $(X,\mu,T_{2},T_{1})$ (in which the roles of the two actions were interchanged) is $\Inv[2]$-sated.
By Proposition~\ref{prop:KI-in-sated} it follows that $\IF{1}\vee\IF{2} = A(X|\IF{1},T_{2}) = A(X|\IF{2},T_{1})$.
By Lemma~\ref{lem:k=2-char-factors} we may assume that the functions $f_{1}$ and $f_{2}$ are measurable with respect to the $\sigma$-algebra $\IF{1}\vee \IF{2}$.
By density and linearity we may assume that $f_{2}=h_{2}^{1}h_{2}^{2}$ and $f_{1}=h_{1}^{1}h_{1}^{2}$ with $h_{\epsilon}^{i}$ being $T_{i}$-invariant.
In this case we have
\[
\iint \prod_{\epsilon\subseteq\{1,2\}} T_{\epsilon}^{\vec g} f_{\epsilon} \dif\Phi_{n}(g_{1}) \dif\Psi_{n}(g_{2})
=
\iint f_{\emptyset}h_{1}^{1}h_{2}^{2} T_{1,2}^{\vec g} (f_{1,2}h_{1}^{2}h_{2}^{1}) \dif\Phi_{n}(g_{1}) \dif\Psi_{n}(g_{2}),
\]
and the conclusion follows from the mean ergodic theorem.
\end{proof}

\section{Furstenberg averages with almost periodic weights}
Recall that the right shift of a function $f$ on $G$ by an element $g$ is defined by $R_{g}f(h):=f(hg)$ and the left shift by $L_{g}f(h):=f(g\inv h)$.
The left and the right shifts are commuting $G$-actions.
A continuous function $f:G\to\C$ is called \emph{almost periodic} if the set $L_{G}R_{G}f=\{L_{g}R_{g'}f : g,g'\in G\}$ is totally bounded with respect to the metric induced by the supremum norm.
The set of continuous almost periodic functions is denoted by $\AP(G)$, it is a closed conjugation invariant subalgebra of the space of bounded continuous functions on $G$.
\begin{lemma}
\label{lem:R-cont-on-AP}
The left and the right shift are jointly continuous on $G\times\AP(G)$.
\end{lemma}
\begin{proof}
We will show that $R$ is jointly continuous at every point $(g_{0},f_{0})\in G\times\AP(G)$, the proof for $L$ is nearly identical.
Let $\epsilon>0$.
By definition of $\AP(G)$ there exists a finite $\epsilon$-dense subset $F\subset L_{G}f_{0}$.
Since $F\subset C(X)$, there exists a neighborhood $U$ of $g_{0}$ such that $|f(g)-f(g_{0})|<\epsilon$ for every $g\in U$ and every $f\in F$.
Let $f'\in\AP(G)$ be such that $\|f'-f_{0}\|_{\infty}<\epsilon$.
Then for every $g'\in G$ and $g\in U$ we have
\begin{multline*}
|R_{g}f'(g') - R_{g_{0}}f_{0}(g')|
\leq
|R_{g}f'(g') - R_{g}f_{0}(g')| + |R_{g}f_{0}(g') - R_{g_{0}}f_{0}(g')|\\
\leq
\epsilon + |L_{(g')\inv}f_{0}(g) - L_{(g')\inv}f_{0}(g_{0})|
\leq
3\epsilon + |f(g) - f(g_{0})|
\leq
4\epsilon
\end{multline*}
for some $f\in F$.
\end{proof}
We recall a consequence of the Peter--Weyl theorem.
\begin{theorem}
\label{thm:ap}
Let $f\in C(G)$.
Then the following conditions are equivalent.
\begin{enumerate}
\item\label{ap:ap} $f\in\AP(G)$.
\item\label{ap:compactification} There exists a compact group $K$ and a continuous homomorphism $\iota:G\to K$ such that $f=f'\circ\iota$ for some $f'\in C(K)$.
\item\label{ap:matrix-coeff} $f$ is a uniform limit of matrix coefficients, that is, functions of the form $g\mapsto\<\pi(g)v,w\>$, where $\pi:G\to U(d)$ is a continuous finite-dimensional representation and $v,w\in\C^{d}$.
\item\label{ap:matrix-coeff-meas} $f$ is a uniform limit of functions of the form $\chi(g) = \<w,\pi(g)v\>$, where $\pi:G\to U(d)$ is a measurable antihomomorphsm and $v,w\in\C^{d}$.
\end{enumerate}
\end{theorem}
\begin{proof}
If \eqref{ap:compactification} holds, then $L_{G}R_{G}f$ is totally bounded as an isometric image of a subset of $L_{K}R_{K}f'$, and this shows \eqref{ap:ap}.
Conversely, if \eqref{ap:ap} holds, then $X:=\overline{R_{G}f}\subset\AP(G)$ is a compact metric space, and by Lemma~\ref{lem:R-cont-on-AP} the $G$-action $R$ is jointly continuous on $X$.
Since $R_{g}$ is isometric for each $g\in G$ and by \cite[\textsection 3, Theorem 2]{MR956049} we obtain a compactification $\iota:G\to K$ and a $K$-action $R'$ on $X$ such that $R_{g}=R'_{\iota(g)}$ for all $g\in G$.
Let $e:X\to\C$ be the evaluation at the identity.
Then $f'(k):=e(R'_{k}f)$ is a continuous function on $K$ that extends $f$.

Suppose now that $f$ is a matrix coefficient associated to a representation $\pi$.
Then $f\in\AP(G)$ since \eqref{ap:compactification} is satisfied with $\iota=\pi$.
Since $\AP(G)$ is closed, this shows that \eqref{ap:matrix-coeff} implies \eqref{ap:ap}.
Conversely, suppose that \eqref{ap:compactification} holds.
Then $f'$ is a uniform limit of matrix coefficients on $K$ by \cite[Theorem 5.11]{MR1397028}.
On the other hand, if $\chi$ is a matrix coefficient on $K$, then $\chi\circ\iota$ is a matrix coefficient on $G$, so we obtain \eqref{ap:matrix-coeff}.

Finally, it is clear that \eqref{ap:matrix-coeff} implies \eqref{ap:matrix-coeff-meas}.
Conversely, every measurable homomorphism $G\to U(d)$ is continuous \cite[Theorem 22.18]{MR551496}, so that \eqref{ap:matrix-coeff-meas} implies \eqref{ap:matrix-coeff}.
\end{proof}

\begin{corollary}
\label{cor:Clim-positive}
For every $\chi\in AP(G)$ the uniform \cesaro{} limit $\UClim_{g} \chi(g)$ exists.
If $\chi$ is positive and not identically zero, then $\UClim_{g} \chi(g)>0$.
\end{corollary}
\begin{proof}
By Theorem~\ref{thm:ap} we have $\chi=f\circ\iota$ for some compactification $\iota:G\to K$ and some $f\in C(K)$.
Since $\iota(G)$ is dense in $K$, the only $G$-invariant measure on $K$ is the Haar measure $\nu$.
Hence $\UClim_{g} \chi_{g} = \UClim_{g} f\circ\iota(g)$ exists and equals $\int_{K} f \dif\nu$ by the ergodic theorem for uniquely ergodic actions.

If $\chi$ is positive and not identically zero, then the same is true of $f$, so $\int_{X} f \dif\nu > 0$ since $\nu$ has full support.
\end{proof}

\begin{proposition}
\label{prop:char-factor-f1}
Let $X\in\MPS[2]$, $f_{1},f_{2}\in L^{\infty}(X)$, and let $\chi\in\AP(G)$.
Assume that $f_{1}\perp A(X|\IF{2},T_{1})$.
Then we have
\[
\UClim_{g} \chi(g) T_{1}^{g} f_{1} T_{1,2}^{g} f_{2} = 0
\qquad\text{in norm of } L^{2}(X).
\]
\end{proposition}
\begin{proof}
Fix a left \reiter{} sequence $F$.
By Theorem~\ref{thm:ap} we may assume that $\chi(g)=\<w,\pi(g)v\>$, where $\pi:G\to U(d)$ is a measurable antihomomorphism and $v,w\in\C^{d}$.
In this case it suffices to prove $\Clim_{g} \pi(g)v \otimes u_{g} = 0$ in $\C^{d}\otimes L^{2}(X)$, where $u_{g} := T_{1}^{g} f_{1} T_{1,2}^{g} f_{2}$.
By the van der Corput inequality (Corollary~\ref{cor:vdC}) it suffices to show that
\[
\liminf_{H} \int_{h} \limsup_{N} \int_{g} \<\pi(g) v,\pi(hg) v\> \int u_{g} u_{hg} \dif\mu \dif F_{N}(g) \dif F'_{H}(h) = 0
\]
for a certain two-sided \reiter{} sequence $F'$.
Since $\pi$ is an antihomomorphism this can be written as
\begin{multline*}
\liminf_{H} \int_{h} \limsup_{N} \int_{g} \<v,\pi(h) v\> \int f_{1} T_{2}^{g}f_{2} T_{1}^{h} f_{1} T_{2}^{g} T_{1,2}^{h} f_{2} \dif\mu \dif F_{N}(g) \dif F'_{H}(h)\\
\leq
\|v\|^{2} \liminf_{H} \int_{h} \limsup_{N} \int_{g} \int f_{1} T_{2}^{g}f_{2} T_{1}^{h} f_{1} T_{2}^{g} T_{1,2}^{h} f_{2} \dif\mu \dif F_{N}(g) \dif F'_{H}(h)\\
=
\|v\|^{2} \liminf_{H} \int_{h} \int f_{1} T_{1}^{h} f_{1} \E(f_{2} T_{1,2}^{h} f_{2}|\IF{2}) \dif\mu \dif F'_{H}(h).
\end{multline*}
Since the conditional expectation is an orthogonal projection, this is bounded by
\[
\|v\|^{2} \|f_{2}\|_{\infty}^{2} \liminf_{H} \int_{h} \| \E(f_{1} T_{1}^{h} f_{1} | \IF{2}) \|_{2} \dif F'_{H}(h).
\]
This vanishes by the assumption.
\end{proof}
\begin{corollary}
\label{cor:conv-k=2}
Let $X\in\MPS[2]$, $f_{1},f_{2}\in L^{\infty}(X)$, and let $\chi\in\AP(G)$.
Then
\[
\UClim_{g} \chi(g) T_{1}^{g} f_{1} T_{1,2}^{g} f_{2}
\quad\text{exists in norm of } L^{2}(X).
\]
\end{corollary}
\begin{proof}
By Theorem~\ref{thm:sated-ext} and Proposition~\ref{prop:KI-in-sated} we may assume that $X$ is magic.
By Proposition~\ref{prop:char-factor-f1} the above limit vanishes if $f_{1}\perp \IF{1}\vee \IF{2}$.
Hence by density and linearity it suffices to consider $f_{1}=h_{1}h_{2}$, where $h_{i}$ is $T_{i}$-invariant.
In this case we have
\[
\chi(g) T_{1}^{g} f_{1} T_{1,2}^{g} f_{2}
=
h_{1} \chi(g) T_{1,2}^{g} (f_{2}h_{2}).
\]
By Theorem~\ref{thm:ap} we may assume that $\chi(g)=\<w,\pi(g)v\>$ for a measurable antihomomorphism $\pi:G\to U(d)$.
The conclusion follows because the map $g\mapsto \pi(g)v \otimes T_{1,2}^{g}(f_{2}h_{2})$ converges in the uniform \cesaro{} sense in $\C^{d}\otimes L^{2}(X)$ by the mean ergodic theorem applied to the antirepresentation $\pi\otimes T_{1,2}$.
\end{proof}

Here and later write $\KF{\epsilon}=\KF{}(T_{\epsilon})=A(X|\mathrm{trivial},T_{\epsilon})$ for the factor spanned by the finite-dimensional $T_{\epsilon}$-invariant subspaces of $L^{2}(X)$ for $\epsilon\subseteq\{1,2\}$.
These factors are used as building blocks for characteristic factors for weighted Furstenberg averages.
\begin{corollary}
\label{cor:char-factors-sated-weighted}
Let $X\in\MPS[2]$, $f_{0},f_{1},f_{2}\in L^{\infty}(X)$, and $\chi\in\AP(G)$.
Suppose that $X$ is magic.
Then
\[
\UClim_{g} \chi(g) f_{0} T_{1}^{g}f_{1} T_{1,2}^{g} f_{2}
=
\UClim_{g} \chi(g) \E(f_{0}|\IF{1} \vee \KF{1,2}) T_{1}^{g}\E(f_{1}|\IF{1} \vee \IF{2}) T_{1,2}^{g} \E(f_{2}|\IF{2} \vee \KF{1,2}).
\]
\end{corollary}
\begin{proof}
By Corollary~\ref{cor:conv-k=2} the uniform \cesaro{} limits on both sides exist in $L^{2}(X)$.
By Proposition~\ref{prop:char-factor-f1} we may assume that $f_{1}$ is $A(X|\IF{1},T_{2})$-measurable, and hence $\IF{1}\vee \IF{2}$-measurable by definition of a magic system.
By density and linearity it suffices to consider $f_{1}=h_{1}h_{2}$, where $h_{j}$ is $T_{j}$-invariant.
In this case we have
\[
\chi(g) f_{0} T_{1}^{g}(h_{1}h_{2}) T_{1,2}^{g} f_{2}
=
\chi(g) f_{0} h_{1} T_{1,2}^{g} (f_{2}h_{2}).
\]
Suppose now that $f_{2} \perp \IF{2} \vee \KF{1,2}$, so that $f_{2}h_{2} \perp\KF{1,2}$ and fix a left \reiter{} sequence $F$.
By Theorem~\ref{thm:ap} we have $\chi(g)=\kappa(\iota(g))$ for some compactification $\iota:G\to K$ and some $\kappa\in C(K)$.
We have a $G$-action on $K$ by left translation by $\iota$.
By Theorem~\ref{thm:A+W} we obtain $\kappa\otimes f_{2}h_{2}\perp A(K\times X|\mathrm{trivial},\iota\times T_{1,2})$.
In particular,
\[
\Clim_{g} \kappa(\iota(g)\cdot)T_{1,2}^{g}(f_{2}h_{2}) = 0
\]
in $L^{2}(K\times X)$.
Passing to a subsequence of our \reiter{} sequence we obtain
\[
\Clim_{g} \kappa(\iota(g) k)T_{1,2}^{g}(f_{2}h_{2}) = 0
\]
in $L^{2}(X)$ for a.e.\ $k\in K$.
By uniform continuity the same actually holds for every $k\in K$, and substituting $k=\id_{K}$ we obtain $\Clim_{g} \chi(g) T_{1,2}^{g} (f_{2}h_{2}) = 0$.
Since this limit does not depend on the subsequence, a subsubsequence argument shows that $\UClim_{g} \chi(g) T_{1,2}^{g} (f_{2}h_{2}) = 0$.

The remaining case $f_{0} \perp \IF{1} \vee \KF{1,2}$ can be handled similarly.
\end{proof}

\section{An almost periodic correlation function and the recurrence theorem}
Before embarking on the proof of our recurrence theorem we state two lemmas that facilitate calculation of integrals.
The first of them concerns relative independence, while the second deals with a certain trilinear form.
\begin{lemma}
\label{lem:I1-I2-indep}
Suppose that $X\in\MPS[2]$.
Then the $\sigma$-algebras $\IF{1}, \IF{2}$ are relatively independent over $\IF{1}\wedge \IF{2}$.
\end{lemma}
\begin{proof}
Let $f_{i}$ be $T_{i}$-invariant and fix a left \reiter{} sequence $F$.
By the mean ergodic theorem for the $G\times G$-action $(T_{1},T_{2})$ we have
\begin{multline*}
\E(f_{1}f_{2}|\IF{1}\wedge \IF{2})
= \lim_{n} \iint T_{1}^{g} T_{2}^{h} f_{1}f_{2} \dif F_{n}(g) \dif F_{n}(h)\\
= \lim_{n} \int T_{2}^{h} f_{1} \dif F_{n}(h) \int T_{1}^{g} f_{2} \dif F_{n}(g)
= \E(f_{1}|\IF{2}) \E(f_{2}|\IF{1}).
\end{multline*}
Since $T_{1}$ and $T_{2}$ commute, the conditional expectation operators onto $\IF{1}$ and $\IF{2}$ commute as well, so the above conditional expectations equal those on $\IF{1}\wedge\IF{2}$.
\end{proof}
Similarly, one can show relative independence over $\IF{1}\wedge \IF{2}$ for the pairs $\IF{1},\IF{1,2}$ and $\IF{2},\IF{1,2}$.
For instance for the first pair we obtain
\begin{multline*}
\E(f_{1}f_{1,2} | \IF{1}\wedge\IF{2} )
= \lim_{n,m} \iint T_{1}^{g} T_{2}^{h} f_{1}f_{1,2} \dif F_{n}(g) \dif F_{m}(h)
= \lim_{n,m} \int T_{2}^{h} f_{1} \int T_{1}^{g}T_{2}^{h} T_{1,2}^{h^{-1}} f_{1,2} \dif F_{m}(g) \dif F_{n}(h)\\
= \lim_{n,m} \int T_{2}^{h} f_{1} \int T_{1}^{h^{-1}g} f_{1,2} \dif F_{m}(g) \dif F_{n}(h)
= \lim_{n} \int T_{2}^{h} f_{1} \E(f_{1,2} | \IF{1}) \dif F_{n}(h)
= \E(f_{1}|\IF{2}) \E(f_{1,2} | \IF{1}).
\end{multline*}

\begin{definition}
The subcategory $\eMPS$ of $\MPS$ consists of ergodic systems, that is, systems for which the $\sigma$-algebra $\IF{1}\wedge\dots\wedge \IF{k}$ is trivial.
\end{definition}
\begin{lemma}
\label{lem:ch3}
Let $X\in\eMPS[2]$ and $f_{0},f_{1},f_{2}\in L^{\infty}(X)$.
Suppose that $f_{0}\in L^{\infty}(\KF{1,2})$, $f_{1}\in L^{\infty}(\IF{1})$, and $f_{2}\in L^{\infty}(\IF{2})$.
Then
\[
\int f_{0} f_{1} f_{2} \dif\mu = \int f_{0} \E(f_{1}|\IF{1}\wedge\KF{2}) \E(f_{2}|\IF{2}\wedge\KF{1}) \dif\mu
\]
\end{lemma}
\begin{proof}
We can clearly replace $f_{0}$ by $\E(f_{0}|\IF{1}\vee\IF{2})$ on both sides.
Since $\IF{1}\vee \IF{2}$ is $T_{1,2}$-invariant, the conditional expectation $\E(\cdot,\IF{1}\vee\IF{2})$ commutes with $T_{1,2}$.
Therefore this conditional expectation maps finite-dimensional $T_{1,2}$-invariant subspaces to finite-dimensional $T_{1,2}$-invariant subspaces.
It follows that $\E(f_{0}|\IF{1}\vee \IF{2})$ is $\KF{1,2}$-measurable, so we may assume that the function $f_{0}$ is measurable with respect to the $\sigma$-algebra $(\IF{1}\vee\IF{2})\wedge\KF{1,2} = A(\IF{1}\vee\IF{2}|\mathrm{trivial},T_{1,2})$.

Since $\IF{1}$ and $\IF{2}$ are independent by Lemma~\ref{lem:I1-I2-indep} we have
\[
(X,\IF{1}\vee \IF{2},T_{1,2})
\cong
(X,\IF{1},T_{1,2}) \times (X,\IF{2},T_{1,2})
=
(X,\IF{1},T_{2}) \times (X,\IF{2},T_{1}),
\]
so that
\[
A(\IF{1}\vee\IF{2}|\mathrm{trivial},T_{1,2})
=
A(\IF{1}|\mathrm{trivial},T_{2}) \otimes A(\IF{2}|\mathrm{trivial},T_{1})
\]
by Theorem~\ref{thm:A+W}.
Therefore we obtain $\int f_{0}f_{1}f_{2} \dif\mu=0$ if $f_{1}\perp \IF{1}\wedge\KF{2}=A(\IF{1}|\mathrm{trivial},T_{2})$ or $f_{2}\perp \IF{2}\wedge\KF{1}=A(\IF{2}|\mathrm{trivial},T_{1})$.
\end{proof}
The next result is central to our approach of establishing a lower bound for weighted ergodic averages.
This is the place where the almost periodic function that will be used to construct the appropriate weight first arises.
\begin{lemma}
\label{lem:approx-by-matrix-coeff}
Let $X\in\eMPS[2]$ and $f_{0},f_{1},f_{2}\in L^{\infty}(X)$.
Suppose $f_{0} \in L^{\infty}(\IF{1}\vee \KF{1,2})$, $f_{1}\in L^{\infty}(\IF{1}\vee \IF{2})$, and $f_{2}\in L^{\infty}(\IF{2}\vee \KF{1,2})$.
Then the function
\[
g\mapsto \int f_{0} T_{1}^{g} f_{1} T_{1,2}^{g} f_{2} \dif\mu
\]
is almost periodic.
\end{lemma}
\begin{proof}
By density and linearity we may assume that $f_{0}=r_{1}r_{1,2}$, $f_{1}=s_{1}s_{2}$, $f_{2}=t_{2}t_{1,2}$, where $r_{1},s_{1}\in L^{\infty}(\IF{1})$, $s_{2},t_{2}\in L^{\infty}(\IF{2})$, and $r_{1,2},t_{1,2}\in L^{\infty}(\KF{1,2})$.
We have
\[
\int r_{1}r_{1,2} T_{1}^{g}(s_{1}s_{2}) T_{1,2}^{g}(t_{2}t_{1,2}) \dif\mu
=
\int r_{1}s_{1} T_{1}^{g}(s_{2}t_{2}) \cdot r_{1,2} T_{1,2}^{g}(t_{1,2}) \dif\mu
\]
By Lemma~\ref{lem:ch3} this equals
\[
\int \E(r_{1}s_{1}|\IF{1}\wedge\KF{2}) \cdot T_{1}^{g} \E(s_{2}t_{2}|\IF{2}\wedge\KF{1}) \cdot r_{1,2} T_{1,2}^{g}(t_{1,2}) \dif\mu
=
\int h_{0} T_{1}^{g} h_{1} T_{1,2}^{g} h_{2} \dif\mu
\]
for some $h_{0}\in L^{\infty}(X)$, $h_{1}\in L^{\infty}(\KF{1})$, and $h_{2}\in L^{\infty}(\KF{1,2})$.
In view of Theorem~\ref{thm:A+W}, by density and linearity we may assume that $h_{1}\in H_{1}$ and $h_{2}\in H_{1,2}$, where $H_{i}$ are finite-dimensional $T_{i}$-invariant subspaces of $L^{2}(X)$.
Let $\{b_{i}^{j}\}_{j}$ be orthonormal bases of $H_{i}$.
Then we have $T_{i}^{g}b_{i}^{j} = \sum_{l} \pi_{i}(g)^{j}_{l}b_{i}^{l}$ with measurable antihomomorphisms $\pi_{i}:G\to U(\dim H_{i})$.
Writing $h_{i}=\sum_{j}a_{i}^{j}b_{i}^{j}$ we obtain
\begin{multline*}
\int h_{0} T_{1}^{g} h_{1} T_{1,2}^{g} h_{2} \dif\mu
=
\int h_{0} \sum_{l,j} \pi_{1}(g)^{j}_{l}a_{1}^{j}b_{1}^{l} \sum_{l',j'} \pi_{1,2}(g)^{j'}_{l'}a_{1,2}^{j'}b_{1,2}^{l'} \dif\mu\\
=
\sum_{l,j,l',j'} \pi_{1}(g)^{j}_{l} \pi_{1,2}(g)^{j'}_{l'} a_{1}^{j} a_{1,2}^{j'} \int h_{0} b_{1}^{l} b_{1,2}^{l'} \dif\mu,
\end{multline*}
and this is a matrix coefficient function.
\end{proof}

Finally, we need to ensure existence of ergodic magic extensions.
\begin{lemma}
\label{lem:ergodic-sated}
Suppose that $X\in\eMPS$ is $\Inv$-sated in $\eMPS$.
Then $X$ is also $\Inv$-sated in $\MPS$.
\end{lemma}
\begin{proof}
Let $(Y,\nu)\to (X,\mu)$ be an extension and let $f'\in L^{\infty}(\Inv Y)$, $f\in L^{\infty}(X)$ be bounded.
Let also $\nu=\int \nu_{y}\dif\nu(y)$ be the ergodic decomposition of $\nu$.
We have $f' = \lim_{n\to\infty} f'_{n}$ in $L^{2}(\nu)$, where each $f'_{n}$ is a finite linear combination of products of bounded $T_{i}$-invariant functions for $i=1,\dots,k$.
Passing to a subsequence we may assume that the same is true in $L^{2}(\nu_{y})$ for almost every $y$, so that $f'\in L^{\infty}(\Inv (Y,\nu_{y}))$ for almost every $y$.

Since almost every $(Y,\nu_{y})$ is an (ergodic) extension of $(X,\mu)$ and by satedness of $X$ in $\eMPS$ this implies that
\[
\int f' f \dif\nu_{y} = \int f' \E(f|\Inv X) \dif\nu_{y}.
\]
Integrating over $y\in Y$ we obtain
\[
\int f' f \dif\nu = \int f' \E(f|\Inv X) \dif\nu,
\]
and, since $f'$ and $f$ were arbitrary, this shows that $\Inv Y$ and $X$ are relatively independent over $\Inv X$.
\end{proof}

\begin{corollary}
\label{cor:ergodic-magic}
Every $X\in\eMPS[2]$ admits an ergodic magic extension.
\end{corollary}
\begin{proof}
The class $\eMPS[2]$ is clearly closed under inverse limits, so we may apply Theorem~\ref{thm:sated-ext} with $\mathcal{C}=\eMPS[2]$.
The resulting system is $\Inv[2]$-sated by Lemma~\ref{lem:ergodic-sated}, so it is magic by Proposition~\ref{prop:KI-in-sated}.
\end{proof}

Now all the tools required for the proof of our main result have been made available.
We proceed with its formulation and proof. 
\begin{theorem}
\label{thm:recurrence}
Let $X\in\eMPS[2]$ and let $f : X \to [0,1]$ be a measurable function.
Then for every $\epsilon>0$ there exists an almost periodic function $\chi : G\to\R_{\geq 0}$ such that $\UClim_{g} \chi(g) = 1$ and
\[
\UClim_{g} \chi(g) \int f T_{1}^{g} f T_{1,2}^{g} f \dif\mu \geq \big( \int f \dif\mu \big)^{4}-\epsilon.
\]
\end{theorem}
The idea to use an almost periodic weight in order to obtain a lower bound for multiple ergodic averages first appeared in the work of Frantzikinakis \cite{MR2415080}.
Since the almost periodic function $\chi$ is necessarily bounded, Theorem~\ref{thm:recurrence} implies in particular that the uniform \cesaro{} limit of $\int f T_{1}^{g} f T_{1,2}^{g} f \dif\mu$ is positive if $f\not\equiv 0$.
\begin{proof}
By Corollary~\ref{cor:ergodic-magic} we may assume that $X$ is magic.
Let $\epsilon>0$ be arbitrary.
By Lemma~\ref{lem:approx-by-matrix-coeff} the function
\[
\kappa(g) := \int \E(f|\IF{1}\vee \KF{1,2}) T_{1}^{g} \E(f|\IF{1}\vee \IF{2}) T_{1,2}^{g} \E(f|\IF{2}\vee \KF{1,2}) \dif\mu
\]
is almost periodic.
Note that we have
\[
\kappa(\id_{G})
\geq
\int f \E(f|\IF{1}\vee \KF{1,2}) \E(f|\IF{1}\vee \IF{2}) \E(f|\IF{2}\vee \KF{1,2}) \dif\mu
\geq
\big( \int f \dif\mu \big)^{4} =: B
\]
by \cite[Lemma 1.6]{MR2794947}.
Let $\phi:\R\to\R_{\geq 0}$ be a continuous function such that $\phi(x)=0$ if $x\leq B-\epsilon$ and $\phi(x)=1$ if $x\geq B$.
Since $\AP(G)$ is a closed algebra and by the Weierstra\ss{} approximation theorem the function $\phi\circ\kappa$ is in $\AP(G)$.
Moreover this function is positive and equals $1$ at the identity, so $\UClim_{g} \phi(\kappa(g))>0$ by Corollary~\ref{cor:Clim-positive}.
Let $\chi := (\UClim_{g} \phi(\kappa(g)))\inv \phi\circ\kappa$.
Note that $\chi\cdot\kappa \geq \chi\cdot(B-\epsilon)$.
By Corollary~\ref{cor:char-factors-sated-weighted} this implies
\[
\UClim_{g} \chi(g) \int f T_{1}^{g} f T_{1,2}^{g} f \dif\mu
=
\UClim_{g} \chi(g) \kappa(g)
\geq
\UClim_{g} \chi(g) (B-\epsilon)
=
B - \epsilon.
\qedhere
\]
\end{proof}
We note that the above proof also yields a generalization of \cite[Theorem 1.3]{MR2794947} to actions of amenable groups.
Indeed, if $\mathcal{X}=\IF{1}\vee \IF{2}$ and $f\in L^{\infty}(X)$, then $\E(f|\IF{1}\vee \IF{2}) = f$ on any extension of $X$.
Therefore \cite[Lemma 1.6]{MR2794947} in fact gives $\kappa(\id) \geq (\int f \dif\mu)^{3}$ in that case.

\begin{proof}[Proof of Theorem~\ref{thm:erg}]
Suppose first that for some $\epsilon>0$ the set $R_{\epsilon}$ is not left syndetic.
Then there exists a left \folner{} sequence $F$ in $G$ none of whose members intersects $R_{\epsilon}$.
Consider the matrix coefficient function $\chi$ given by Theorem~\ref{thm:recurrence} with $f=1_{A}$ and $\epsilon/2$ in place of $\epsilon$.
By the assumption we have
\[
\chi(g) \int f T_{1}^{g} f T_{1,2}^{g} f \leq \chi(g) \big(\big( \int f \dif\mu \big)^{4} - \epsilon \big)
\]
for every $g\in F_{N}$ for every $N$, contradicting the conclusion of Theorem~\ref{thm:recurrence}.

In order to see that $R_{\epsilon}$ is also right syndetic it suffices to notice that
\[
R_{\epsilon}\inv
= \{g : \mu(A\cap T_{1}^{g\inv}A \cap T_{1,2}^{g\inv}A) \geq \mu(A)^{4}-\epsilon\}
= \{g : \mu(T_{1,2}^{g}A\cap T_{2}^{g}A \cap A) \geq \mu(A)^{4}-\epsilon\},
\]
and this set is left syndetic by the above argument with the roles of $T_{1}$ and $T_{2}$ reversed.
\end{proof}

\section{Combinatorial application}
Since the lower bound in our multiple recurrence theorem only holds for ergodic systems, we need an appropriate version of the Furstenberg correspondence principle.

Recall that a point $x$ of a compact metric space $X$ is called \emph{quasi-generic} with respect to a probability measure $\mu$ for a continuous action $T$ of $G$ on $X$ if there exists a left \folner{} sequence such that for every $f\in C(X)$ we have $\int f \dif\mu = \lim_{N} \int f(T^{g}x) \dif F_{N}(g)$.
It follows from the mean ergodic theorem that if $\mu$ is ergodic, then $\mu$-a.e.\ point is quasi-generic.
As pointed out by Furstenberg, this implies the following version of \cite[Proposition 3.9]{MR603625} with identical proof.
\begin{proposition}
\label{prop:dense-quasi-generic}
Let $X$ be a compact metric space and $T$ a continuous $G$-action.
Then every point $x_{0}\in X$ such that $\overline{T^{G}x_{0}}=X$ is quasi-generic for every ergodic invariant probability measure on $X$.
\end{proposition}
This can be used to prove a version of the Furstenberg correspondence principle \cite[Theorem 4.17]{MR1776759} that provides an ergodic measure-preserving system.
\begin{lemma}
\label{lem:correspondence}
Let $G$ be a countable amenable group and $E\subset G$.
Then there exists an ergodic measure-preserving system $(X,\mu,T)$ and a measurable subset $A\subset X$ with $\mu(A)\geq\ud(E)$ such that
\[
\mu(T^{g_{1}}A\cap\dots\cap T^{g_{k}}A) \leq \ud(g_{1}E \cap\dots\cap g_{k}E)
\]
for any $k\in\N$ and $g_{1},\dots,g_{k}\in G$.
\end{lemma}
The $G=\Z$ case of Lemma~\ref{lem:correspondence} was first used by Bergelson, Host, and Kra \cite[Proposition 3.1]{MR2138068}, who got the idea from Lesigne.
\begin{proof}
If $\ud(E)=0$, then we can consider a one-point system $X$ and take $A$ to be the empty set.
Thus we may assume that $E$ has positive upper Banach density.

Consider $X'=\{0,1\}^{G}$ with the product topology and the left $G$-action $(T^{g}x)_{h}=x_{hg}$.
Let $e\in X'$ be the indicator function of $E$ and set $X:=\overline{T^{G}e}$.
Set also $A := \{x\in X:x_{\id}=1\}$.
Let $F$ be a left \folner{} sequence in $G$ such that $\lim_{N} \frac{|E\cap F_{N}|}{|F_{N}|}=\ud (E)$.
Passing to a subsequence we may assume that the sequence of measures $(\int \delta_{T^{g}e} \dif F_{N}(g))_{N}$ converges weakly.
Its limit $\nu$ is a $T$-invariant probability measure supported on $X$ such that $\nu(A)=\ud(E)$.
By the ergodic decomposition there exists an ergodic $T$-invariant probability measure $\mu$ on $X$ such that $\mu(A)\geq\nu(A)$.

By Proposition~\ref{prop:dense-quasi-generic} the point $e$ is quasi-generic for $\mu$.
Let $\Phi$ be a left \folner{} sequence that witnesses the quasi-genericity.
Since the set $T^{g_{1}}A\cap\dots\cap T^{g_{k}}A$ is clopen, we have
\begin{multline*}
\mu(T^{g_{1}}A\cap\dots\cap T^{g_{k}}A)
= \lim_{N} \int 1_{T^{g_{1}}A\cap\dots\cap T^{g_{k}}A}(T^{g}e) \dif\Phi_{N}(g)
= \lim_{N} \int \prod_{i=1}^{k}1_{A}(T^{g_{i}\inv g}e) \dif\Phi_{N}(g)\\
= \lim_{N} |\Phi_{N}|\inv |\{g\in\Phi_{N}:g_{i}\inv g\in E,i\leq k\}|
= \lim_{N} |\Phi_{N}|\inv |\Phi_{N}\cap g_{1} E \cap \dots \cap g_{k}E|
\leq \ud (g_{1} E \cap \dots \cap g_{k}E).
\end{multline*}
\end{proof}

\begin{proof}[Proof of Theorem~\ref{thm:combi}]
Let $(X,\mu,T)$ and $A\subset X$ be the ergodic system and the measurable subset obtained by applying Lemma~\ref{lem:correspondence} to $E\subset G\times G$.
Writing $T_{1}^{g}:=T^{(g,\id)}$ and $T_{2}^{g}:=T^{(\id,g)}$ we have $\ud(E\cap (g,\id)E\cap (g,g)E) \geq \mu(A\cap T_{1}^{g}A\cap T_{1,2}^{g}A)$.
By Theorem~\ref{thm:erg} the latter quantity is bounded below by $\mu(A)^{4}-\epsilon$ for a set of $g$ that is both left and right syndetic.
On the other hand, $\mu(A)^{4}-\epsilon \geq \ud(E)^{4}-\epsilon$, and we obtain the claim.
\end{proof}

\appendix

\section{Mean convergence for $k=3$}
In this appendix we give a new proof of the $k=3$ case of \cite[Theorem 1.1(2)]{arxiv:1111.7292} using the machinery of sated extensions.
\begin{proposition}
\label{prop:conv-k=3}
Suppose that $X\in\MPS[3]$ and $f_{1},f_{2},f_{3}\in L^{\infty}(X)$.
Then the limit
\[
\UClim_{g} T_{1}^{g}f_{1} T_{1,2}^{g}f_{2} T_{1,2,3}^{g}f_{3}
\]
exists in $L^{2}(X)$.
\end{proposition}
In order to use a satedness argument we have to construct an interesting extension of a system $X\in\MPS[3]$.
We consider the \emph{Furstenberg coupling} $X^{F}$, which consists of the following data.
The base space is $X^{3}$ and the measure is given by
\[
\int f_{0} \otimes f_{1} \otimes f_{2} \dif\mu^{F} := \UClim_{g} \int f_{0} T_{2}^{g}f_{1} T_{2,3}^{g} f_{2} \dif\mu.
\]
This limit exists by Corollary~\ref{cor:conv-k=2}, and Lemma~\ref{lem:coupling} shows that this does determine a unique measure on $X^{k+1}$.
We have the following $G$-actions on $X^{F}$:
\[
T_{F1}=T_{1}\times T_{1}\times T_{1},
\quad
T_{F2}=\Id \times T_{2}\times T_{2,3},
\quad
T_{F3}=T_{2,3}\times T_{3}\times \Id.
\]
The action $T_{F1}$ preserves the measure $\mu^{F}$ since $T_{1}$ commutes with $T_{2}$ and $T_{3}$.
The action $T_{F2}$ preserves the measure $\mu^{F}$ by left invariance of the uniform \cesaro{} limit and the action $T_{F3}$ by right invariance of the uniform \cesaro{} limit.
Moreover, these actions clearly commute.
Hence we see that $X^{F}$ is an extension of $X$ under the projection $\pi$ onto the second coordinate.

We also need a special case of a lemma from \cite{robertson} which in turn generalizes \cite[Lemma 4.7]{MR2599882}.
\begin{lemma}
\label{lem:lift}
Let $X\in\MPS[3]$ and $f_{1},f_{2},f_{3}$ be bounded functions on $X$.
Suppose that
\[
\E_{\mu^{F}}(f_{1}\otimes f_{2}\otimes f_{3} | \IF{F1,F2}) = 0.
\]
Then
\[
\UClim_{g} T_{1}^{g} f_{1} T_{1,2}^{g}f_{2} T_{1,2,3}^{g} f_{3} = 0
\qquad \text{in norm of } L^{2}(X).
\]
\end{lemma}
\begin{proof}
Fix a left \reiter{} sequence $F$.
By Corollary~\ref{cor:vdC} it suffices to show that
\[
0 = \liminf_{n} \limsup_{m} \iiint T_{1}^{g} f_{1} T_{1,2}^{g} T_{1,2,3}^{g} f_{3} \cdot T_{1}^{hg} f_{1} T_{1,2}^{hg} f_{2} T_{1,2,3}^{hg} f_{3} \dif\mu \dif F'_{n}(h) \dif F_{m}(g).
\]
This expression equals
\begin{multline*}
\liminf_{n} \limsup_{m} \iiint
T_{1}^{g} f_{1} T_{1,2}^{g} f_{2} T_{1,2,3}^{g} f_{3} \cdot T_{1}^{g}T_{1}^{h} f_{1} T_{1,2}^{g}T_{1,2}^{h}f_{2} T_{1,2,3}^{g}T_{1,2,3}^{h} f_{3}
\dif\mu \dif F'_{n}(h) \dif F_{m}(g)\\
=
\liminf_{n} \limsup_{m} \iiint
f_{1} T_{2}^{g} f_{2} T_{2,3}^{g} f_{3} \cdot T_{1}^{h} f_{1} T_{2}^{g}T_{1,2}^{h}f_{2} T_{2,3}^{g}T_{1,2,3}^{h} f_{3}
\dif\mu \dif F'_{n}(h) \dif F_{m}(g)\\
=
\liminf_{n} \iint
f_{1}T_{1}^{h} f_{1} \otimes f_{2}T_{1,2}^{h}f_{2} \otimes f_{3} T_{1,2,3}^{h} f_{3}
\dif\mu^{F} \dif F'_{n}(h).
\end{multline*}
By the mean ergodic theorem this equals
\[
\| \E_{\mu^{F}}(f_{1}\otimes f_{2}\otimes f_{3} | \IF{F1,F2}) \|_{L^{2}(\mu^{F})}^{2} = 0.
\qedhere
\]
\end{proof}

\begin{proof}[Proof of Proposition~\ref{prop:conv-k=3}]
Consider the idempotent subclass $\mathcal{J}$ of $\MPS[3]$ that consists of the systems with $\mathcal{X}=\IF{2} \vee \IF{3} \vee \IF{1,2}$.
By Theorem~\ref{thm:sated-ext} we may assume that $X$ is $\mathcal{J}$-sated.
Assume now that $X\in\MPS[3]$ is $\mathcal{J}$-sated and $f_{2} \perp \IF{2} \vee \IF{3} \vee \IF{1,2}$.
By satedness we have $1\otimes f_{2}\otimes 1 \perp \IF{F2} \vee \IF{F3} \vee \IF{F1,F2}$.
Moreover, for any $f_{1},f_{3}\in L^{\infty}(X)$ the function $f_{1}\otimes 1\otimes 1$ is $\IF{F2}$-measurable and the function $1\otimes 1\otimes f_{3}$ is $\IF{F3}$-measurable.
Therefore $f_{1}\otimes f_{2}\otimes f_{3} \perp \IF{F1,F2}$.
By Lemma~\ref{lem:lift} this implies
\[
\UClim_{g} T_{1}^{g}f_{1} T_{1,2}^{g}f_{2} T_{1,2,3}^{g}f_{3} = 0.
\]
Hence we may assume that $f_{2}$ is $\IF{2}\vee \IF{3}\vee \IF{1,2}$-measurable.
By density and linearity it suffices to consider $f_{2} = h_{2}h_{3}h_{1,2}$, where $h_{i}$ is $T_{i}$-invariant.
In this case we have
\[
T_{1}^{g}f_{1} T_{1,2}^{g} f_{2} T_{1,2,3}^{g}f_{3}
=
h_{1,2} T_{1}^{g}(f_{1}h_{2}) T_{1,2,3}^{g}(f_{3}h_{3}),
\]
and this converges in the uniform \cesaro{} sense by Corollary~\ref{cor:conv-k=2} applied to the actions $T_{1}$ and $T_{2,3}$.
\end{proof}

The main difficulty in extending this approach to the multiple ergodic theorem to $k>3$ commuting actions consists in explicitly constructing a suitable extension, which can be used to exploit satedness, of a given action.
Note that for discrete groups a new tool for constructing such extensions appeared in Austin's proof, see \cite[Theorem 2.1]{arXiv:1309.4315}.

\section{Small correlation sequences}
The lower bound in the ergodic Roth theorem \cite[Theorem 1.2]{MR2138068} might suggest that $\Clim_{n} \mu(A\cap T^{n}A\cap T^{2n}A) \geq \mu(A)^{3}$ for any measure-preserving transformation $(X,T)$ and any $A\subset X$.
The following counterexample shows that this is not the case.

Consider an irrational rotation $T$ by $\alpha$ on the circle $X = \R\mod 1$ and let $B\subset X$ be an interval of length $\delta\leq 1/3$.
Then the function $m(x):=\mu(B-x \cup B \cup B+x)$ has the following form.

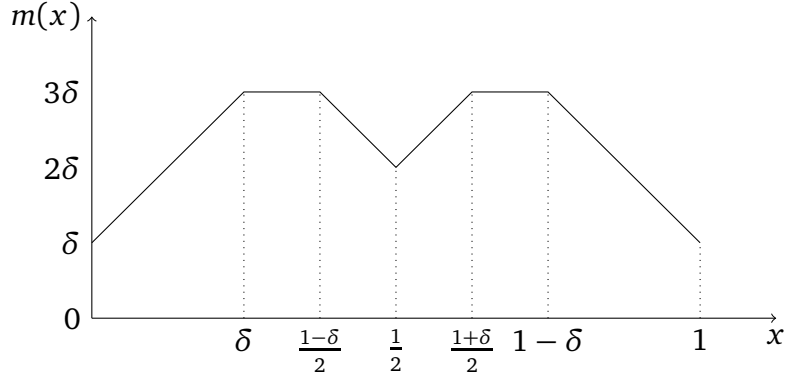
\begin{center}
\begin{tikzpicture}
\draw[->] (0,0) -- (9,0) node[anchor=north] {$x$};
\draw[->] (0,0) -- (0,4) node[anchor=east] {$m(x)$};
\draw (0,1) -- (2,3) -- (3,3) -- (4,2) -- (5,3) -- (6,3) -- (8,1);
\draw[dotted] (2,0) node[anchor=north] {$\delta$} -- (2,3);
\draw[dotted] (3,0) node[anchor=north] {$\frac{1-\delta}{2}$} -- (3,3);
\draw[dotted] (4,0) node[anchor=north] {$\frac12$} -- (4,2);
\draw[dotted] (5,0) node[anchor=north] {$\frac{1+\delta}{2}$} -- (5,3);
\draw[dotted] (6,0) node[anchor=north] {$1-\delta$} -- (6,3);
\draw[dotted] (8,0) node[anchor=north] {$1$} -- (8,1);
\draw (0,0) node[anchor=east] {$0$} (0,1) node[anchor=east] {$\delta$} (0,2) node[anchor=east] {$2\delta$} (0,3) node[anchor=east] {$3\delta$};
\end{tikzpicture}
\end{center}

Let now $A:=X\setminus B$.
Then
\begin{multline*}
\lim_{N} \frac1N \sum_{n=1}^{N} \mu(A\cap T^{n}A \cap T^{2n}A)
=
\lim_{N} \frac1N \sum_{n=1}^{N} (1-m(n\alpha))
=
1 - \int_{0}^{1}m\\
=
1-(3\delta - \frac52 \delta^{2})
=
(1-\delta)^{3} - \delta^{2}/2 + \delta^{3}
<
\mu(A)^{3}.
\end{multline*}

We will now describe an example that shows that the exponent in Theorem~\ref{thm:recurrence} cannot be improved to 3.
The construction is based on \cite[Theorem 1.2]{MR2794947}, but the function $f$ that appears in the proof has been optimized numerically to maximize the exponent (however, we do not claim that this exponent is the best possible).
As a pleasant side effect the optimal function has a particularly simple form.
\begin{theorem}
\label{thm:small-correlation}
For every countable amenable group $G$ there exist a system $X\in\eMPS[2]$ such that the group generated by $T_1$ and $T_2$ acts weakly mixingly and a measurable set $A$ with $0<\mu(A)<1$ such that 
\[
\mu(A\cap T_1^g A\cap T_{1,2}^g A)< \mu(A)^{3.19}
\]
for all $g\neq id_G$.
\end{theorem}
\begin{proof}
Consider $Y:=\{0, 1, 2\}^G$ and denote the $(\frac13,\frac13,\frac13)$-Bernoulli measure on $Y$ by $\nu$.
The two natural left $G$-actions on $Y$ are the right Bernoulli shift $(R^g y)_h=y_{hg}$ and the left Bernoulli shift $(L^{g} y)_h=y_{g^{-1}h}$.
Note that these actions preserve $\nu$ and commute.
Therefore
\[
(X, \mu, T_1, T_2) := (Y\times Y\times Y, \nu\times\nu\times \nu, R\times \Id \times R, L\times R\times \Id )
\]
is a system in $\MPS[2]$.
Since the actions $L$ and $R$ are weakly mixing, it follows that $T_{1}$ and $T_{2}$ span a weakly mixing group action on $X$, and in particular we obtain $X\in\eMPS[2]$.

For $i, j,k\in \{0, 1, 2\}$ let
\[
f(i, j, k) :=
\begin{cases}
1 & \text{if $i,j,k$ are pairwise different},\\
0 & \text{otherwise}.
\end{cases}
\]
Let $F$ be the function on $X$ defined by $F(y, z, w)=f(y_{id_G}, z_{id_G}, w_{id_G})$.
The function $F$ is the indicator function of a measurable subset $A\subset X$ with $\mu(A)=3!/3^{3}$.
For every $g\neq id_G$ we have
\begin{multline*}
\mu(A\cap T_1^g A\cap T_{1,2}^g A)=\int F\cdot T_1^{g^{-1}}F \cdot T_{1,2}^{g^{-1}} F\, \dif\mu\\
=\int F(y, z, w)\cdot F(R^{g^{-1}}y, z, R^{g^{-1}}w)\cdot F(R^{g^{-1}}L^{g^{-1}}y, R^{g^{-1}}z, R^{g^{-1}}w)\, \dif\mu(y,z,w)\\
=\int f(y_{id_G}, z_{id_G}, w_{id_G})\cdot  f(y_{g^{-1}}, z_{id_G}, w_{g^{-1}}) \cdot  f(y_{id_G}, z_{g^{-1}}, w_{g^{-1}})\, \dif\nu(y) \dif\nu(z) \dif\nu(w).
\end{multline*}
Since $\nu$ is a product measure, this equals
\begin{multline*}
\frac {1}{3^6}\sum_{i, j,k, i',j' k'} f(i, j, k')f(i',j, k)f(i, j', k)\\
=
\frac{1}{3^{6}} \sum_{i, j,k} (\sum_{k' : \{i,j,k'\}=\{0,1,2\}} 1) (\sum_{i' : \{i',j,k\}=\{0,1,2\}} 1) (\sum_{j' : \{i,j',k\}=\{0,1,2\}} 1)\\
=
\frac{1}{3^{6}} \sum_{i, j,k} \delta_{i\neq j} \delta_{j\neq k} \delta_{i\neq k}
=
\frac{3!}{3^6}
<
\left( \frac{3!}{3^{3}} \right)^{3.19}
=
\mu(A)^{3.19}.
\qedhere
\end{multline*}
\end{proof}

\printbibliography
\end{document}